\documentclass[a4paper,twoside]{article}
\usepackage{a4}
\usepackage{amssymb}
\usepackage{amsmath}
\usepackage{upref}
\usepackage[colorlinks,citecolor=blue,linkcolor=blue]{hyperref}
\usepackage[dvipsnames]{color}
\usepackage[active]{srcltx}
\allowdisplaybreaks[2] 
\newcount\minutes \newcount\hours
\hours=\time \divide\hours 60 \minutes=\hours
 \multiply\minutes-60
\advance\minutes \time
\newcommand{\klockan}{\the\hours:{\ifnum\minutes<10 0\fi}\the\minutes}
\newcommand{\tid}{\today\ \klockan}
\newcommand{\prtid}{\smash{\raise 10mm \hbox{\LaTeX ed \tid}}}
\renewcommand{\prtid}{}
\makeatletter  \headheight 10pt
\def\sectionmark#1{} 
\def\subsectionmark#1{}
\newcommand{\sectnr}{\ifnum \c@secnumdepth >\z@
                 \thesection.\hskip 1em\relax \fi}
\def\@evenhead{\footnotesize\rm\thepage\hfil\leftmark\hfil\llap{\prtid}}
\def\@oddhead{\footnotesize\rm\rlap{\prtid}\hfil\rightmark\hfil\thepage}
\def\tableofcontents{\section*{Contents} 
 \@starttoc{toc}}
\makeatother
\makeatletter
\def\@biblabel#1{#1.}
\makeatother
\makeatletter
\let\Thebibliography=\thebibliography
\renewcommand{\thebibliography}[1]{\def\@mkboth##1##2{}\Thebibliography{#1}
\addcontentsline{toc}{section}{References}
\frenchspacing 
\setlength{\@topsep}{0pt}
\setlength{\itemsep}{0pt}%
\setlength{\parskip}{0pt plus 2pt}%
} \makeatother
\makeatletter
\def\mdots@{\mathinner.\nonscript\!.%
 \ifx\next,.\else\ifx\next;.\else\ifx\next..\else
 \nonscript\!\mathinner.\fi\fi\fi}
\let\ldots\mdots@
\let\cdots\mdots@
\let\dotso\mdots@
\let\dotsb\mdots@
\let\dotsm\mdots@
\let\dotsc\mdots@
\def\vdots{\vbox{\baselineskip2.8\p@ \lineskiplimit\z@
    \kern6\p@\hbox{.}\hbox{.}\hbox{.}\kern3\p@}}
\def\ddots{\mathinner{\mkern1mu\raise8.6\p@\vbox{\kern7\p@\hbox{.}}%
    \raise5.8\p@\hbox{.}\raise3\p@\hbox{.}\mkern1mu}}
\makeatother
\makeatletter
\let\Enumerate=\enumerate
\renewcommand{\enumerate}{\Enumerate%
\setlength{\@topsep}{0pt}
\setlength{\itemsep}{0pt}%
\setlength{\parskip}{0pt plus 1pt}%
\renewcommand{\theenumi}{\textup{(\alph{enumi})}}%
\renewcommand{\labelenumi}{\theenumi}%
}
\let\endEnumerate=\endenumerate
\renewcommand{\endenumerate}{\endEnumerate\ifhmode\unskip\fi}
\makeatother
\makeatletter
\def\@seccntformat#1{\csname the#1\endcsname.\quad}
\makeatother
\newcommand{\authortitle}[2]{\author{#1}\title{#2}\markboth{#1}{#2}}
\newcommand{\art}[6]{{\sc #1, \rm #2, \it #3 \bf #4 \rm (#5), \mbox{#6}.}}
\newcommand{\auth}[2]{{#1, #2.}}
\newcommand{\artprep}[3]{{\sc #1, \rm #2, #3.}}
\newcommand{\arttoappear}[3]{{\sc #1, \rm #2, to appear in \it #3}}
\newcommand{\book}[3]{{\sc #1, \it #2, \rm #3.}}
\newcommand{\AND}{{\rm and }}
\RequirePackage{amsthm}
\newtheoremstyle{descriptive}%
  {\topsep}   
  {\topsep}   
  {\rmfamily} 
  {}          
  {\bfseries} 
  {.}         
  { }         
  {}          
\newtheoremstyle{propositional}%
  {\topsep}   
  {\topsep}   
  {\itshape}  
  {}          
  {\bfseries} 
  {.}         
  { }         
  {}          
\theoremstyle{propositional}
\newtheorem{thm}{Theorem}[section]
\newtheorem{prop}[thm]{Proposition}
\newtheorem{lem}[thm]{Lemma}
\newtheorem{lemma}[thm]{Lemma} 
\newtheorem{cor}[thm]{Corollary}
\theoremstyle{descriptive}
\newtheorem{deff}[thm]{Definition}
\newtheorem{example}[thm]{Example}
\newtheorem{remark}[thm]{Remark}
\newtheorem{openprob}[thm]{Open problem}
\makeatletter
\renewenvironment{proof}[1][\proofname]{\par
  \pushQED{\qed}%
  \normalfont
  \trivlist
  \item[\hskip\labelsep
        \itshape
    #1\@addpunct{.}]\ignorespaces
}{%
  \popQED\endtrivlist\@endpefalse
} \makeatother
\newcommand{\setm}{\setminus}
\renewcommand{\emptyset}{\varnothing}
\def\vint{\mathop{\mathchoice%
          {\setbox0\hbox{$\displaystyle\intop$}\kern 0.22\wd0%
           \vcenter{\hrule width 0.6\wd0}\kern -0.82\wd0}%
          {\setbox0\hbox{$\textstyle\intop$}\kern 0.2\wd0%
           \vcenter{\hrule width 0.6\wd0}\kern -0.8\wd0}%
          {\setbox0\hbox{$\scriptstyle\intop$}\kern 0.2\wd0%
           \vcenter{\hrule width 0.6\wd0}\kern -0.8\wd0}%
          {\setbox0\hbox{$\scriptscriptstyle\intop$}\kern 0.2\wd0%
           \vcenter{\hrule width 0.6\wd0}\kern -0.8\wd0}}%
          \mathopen{}\int}
\newcommand{\Cp}{{C_p}}
\DeclareMathOperator{\capp}{cap}
\newcommand{\cp}{\capp_p}
\newcommand{\grad}{\nabla}
\DeclareMathOperator{\diam}{diam} 
\DeclareMathOperator{\spt}{supp}

 \newcommand{\loc}{_{\rm loc}}
{\catcode`p =12 \catcode`t =12 \gdef\eeaa#1pt{#1}}      
\def\accentadjtext#1{\setbox0\hbox{$#1$}\kern   
                \expandafter\eeaa\the\fontdimen1\textfont1 \ht0 }
\def\accentadjscript#1{\setbox0\hbox{$#1$}\kern 
                \expandafter\eeaa\the\fontdimen1\scriptfont1 \ht0 }
\def\accentadjscriptscript#1{\setbox0\hbox{$#1$}\kern   
                \expandafter\eeaa\the\fontdimen1\scriptscriptfont1 \ht0 }
\def\accentadjtextback#1{\setbox0\hbox{$#1$}\kern       
                -\expandafter\eeaa\the\fontdimen1\textfont1 \ht0 }
\def\accentadjscriptback#1{\setbox0\hbox{$#1$}\kern     
                -\expandafter\eeaa\the\fontdimen1\scriptfont1 \ht0 }
\def\accentadjscriptscriptback#1{\setbox0\hbox{$#1$}\kern 
                -\expandafter\eeaa\the\fontdimen1\scriptscriptfont1 \ht0 }
\def\itoverline#1{{\mathsurround0pt\mathchoice
        {\rlap{$\accentadjtext{\displaystyle #1}
                \accentadjtext{\vrule height1.593pt}
                \overline{\phantom{\displaystyle #1}
                \accentadjtextback{\displaystyle #1}}$}{#1}}
        {\rlap{$\accentadjtext{\textstyle #1}
                \accentadjtext{\vrule height1.593pt}
                \overline{\phantom{\textstyle #1}
                \accentadjtextback{\textstyle #1}}$}{#1}}
        {\rlap{$\accentadjscript{\scriptstyle #1}
                \accentadjscript{\vrule height1.593pt}
                \overline{\phantom{\scriptstyle #1}
                \accentadjscriptback{\scriptstyle #1}}$}{#1}}
        {\rlap{$\accentadjscriptscript{\scriptscriptstyle #1}
                \accentadjscriptscript{\vrule height1.593pt}
                \overline{\phantom{\scriptscriptstyle #1}
                \accentadjscriptscriptback{\scriptscriptstyle #1}}$}{#1}}}}
\newcommand{\dmu}{d\mu}
\newcommand{\ga}{\gamma}
\newcommand{\gat}{\tilde{\gamma}}
\newcommand{\Om}{\Omega}
\newcommand{\Ga}{\Gamma}
\renewcommand{\phi}{\varphi}
\newcommand{\p}{{$p\mspace{1mu}$}}
\newcommand{\R}{\mathbf{R}}
\newcommand{\Q}{\mathbf{Q}}
\newcommand{\eR}{{\overline{\R}}}
\newcommand{\gt}{{\tilde{g}}}
\newcommand{\ut}{{\tilde{u}}}
\newcommand{\B}{{\cal B}}
\newcommand{\Vt}{{\widetilde{V}}}
\newcommand{\Ut}{{\widetilde{U}}}
\newcommand{\Bt}{{\widetilde{\B}}}
\newcommand{\limplus}{{\mathchoice{\vcenter{\hbox{$\scriptstyle +$}}}
  {\vcenter{\hbox{$\scriptstyle +$}}}
  {\vcenter{\hbox{$\scriptscriptstyle +$}}}
  {\vcenter{\hbox{$\scriptscriptstyle +$}}}
}}
\newcommand{\Np}{N^{1,p}}
\newcommand{\Nploc}{N^{1,p}\loc}
\newcommand{\Wploc}{W^{1,p}\loc}
\newcommand{\Npploc}{N^{1,p}_{\textup{fine-loc}}}
\newcommand{\Npfloc}{N^{1,p}_{\textup{f-loc}}}
\newcommand{\Npqloc}{N^{1,p}_{\textup{quasi-loc}}}
\newcommand{\Dp}{D^p}
\newcommand{\Dploc}{D^{p}\loc}
\newcommand{\Dpploc}{D^{p}_{\textup{fine-loc}}}
\newcommand{\Dpqloc}{D^{p}_{\textup{quasi-loc}}}
\newcommand{\Lploc}{L^{p}\loc}
\newcommand{\Lpploc}{L^{p}_{\textup{fine-loc}}}
\newcommand{\Lpqloc}{L^{p}_{\textup{quasi-loc}}}
\newcommand{\Wp}{W^{1,p}}
\numberwithin{equation}{section}
\newcommand{\hNp}{\widehat{N}^{1,p}}
\DeclareMathOperator{\Mod}{Mod}
\newcommand{\Modp}{{\Mod_p}}
\newcommand{\be}{\beta}

\newcommand{\eqv}{\ensuremath{
\mathchoice{\quad \Longleftrightarrow \quad}{\Leftrightarrow}
                {\Leftrightarrow}{\Leftrightarrow}} }
\newcommand{\imp}{\ensuremath{\Rightarrow} }

\newenvironment{ack}{\medskip{\it Acknowledgement.}}{}

\begin{document}

\authortitle{Anders Bj\"orn, Jana Bj\"orn and Visa Latvala}
{Sobolev spaces,  fine gradients and quasicontinuity on quasiopen sets}
\title{Sobolev spaces,  fine gradients and quasicontinuity on quasiopen sets
in $\R^n$ and metric spaces}
\author{
Anders Bj\"orn \\
\it\small Department of Mathematics, Link\"oping University, \\
\it\small SE-581 83 Link\"oping, Sweden\/{\rm ;}
\it \small anders.bjorn@liu.se
\\
\\
Jana Bj\"orn \\
\it\small Department of Mathematics, Link\"oping University, \\
\it\small SE-581 83 Link\"oping, Sweden\/{\rm ;}
\it \small jana.bjorn@liu.se
\\
\\
Visa Latvala \\
\it\small Department of Physics and Mathematics,
University of Eastern Finland, \\
\it\small P.O. Box 111, FI-80101 Joensuu,
Finland\/{\rm ;}
\it \small visa.latvala@uef.fi
}

\date{}

\maketitle

\noindent{\small
 {\bf Abstract}.
We study different definitions of Sobolev spaces on quasiopen sets 
in a complete metric space $X$ equipped with a doubling measure
supporting a \p-Poincar\'e inequality with $1<p<\infty$,
and connect them to the Sobolev theory in $\R^n$.
In particular, we show that for quasiopen subsets of $\R^n$ 
the Newtonian functions, which are 
naturally defined in any metric space,
coincide with the quasicontinuous representatives of the Sobolev functions 
studied by Kilpel\"ainen and Mal\'y in 1992. 
As a by-product, we establish the quasi-Lindel\"of principle of the 
fine topology in metric spaces 
and study several variants of local Newtonian and Dirichlet spaces 
on quasiopen sets. }

\bigskip


\medskip


\section{Introduction}

We study different
definitions of Sobolev functions on nonopen sets in metric spaces.
Even in $\R^n$, it is not obvious how to define Sobolev spaces
on nonopen subsets, but fruitful definitions have been given
on quasiopen sets~$U$,
i.e.\ on sets which differ from open sets by sets of arbitrarily small
capacity.
Kilpel\"ainen--Mal\'y~\cite{KiMa92} gave the first definition of $\Wp(U)$
in 1992 by means of quasicovering patches of global Sobolev functions.
They also defined a generalized so-called fine gradient for
functions in $\Wp(U)$.
More recently, Sobolev spaces, and in particular the
so-called Newtonian spaces,
have been studied on metric spaces.
Thus by considering $U$ as a metric space in its own right
(and forgetting the ambient space) we get another candidate,
the Newtonian space $\Np(U)$.

The purpose of this paper is to show that the theory of Sobolev functions 
nicely
extends to quasiopen sets in the 
metric setting. In particular, we show that for quasiopen sets $U\subset\R^n$,  the two spaces $\Wp(U)$ and $\Np(U)$ coincide, with equal norms.
To be precise, the functions in $\Np(U)$ are more exactly defined than a.e.,
and we have the following result.

\begin{thm} \label{thm-intro-equiv-Sob}
Let $U \subset \R^n$ be quasiopen,
and let $u:U \to \eR:=[-\infty,\infty]$ be an everywhere defined function.
Then $u \in \Np(U)$ if and only if $u \in \Wp(U)$ and $u$ is quasicontinuous.
Moreover, in this case $\|u\|_{\Np(U)}=\|u\|_{\Wp(U)}$.
\end{thm}

On open sets in $\R^n$, the equality between the Newtonian 
and Sobolev spaces was proved by Shanmugalingam~\cite{Sh-rev}.
The proof of Theorem~\ref{thm-intro-equiv-Sob}
is quite involved, and we will need most of
the results in this paper to deduce it.
We will also use several results related to the fine topology
from our earlier papers \cite{BBnonopen}--\cite{BBL2}.

In $\R^n$ (and in the setting of this paper), quasiopen sets appear whenever truncations of Sobolev functions are considered. This is so because quasiopen sets coincide with the superlevel sets of suitable (quasicontinuous) representatives of the global Sobolev functions,
see Fuglede~\cite{Fugl71},
Kilpel\"ainen--Mal\'y~\cite{KiMa92}
and Proposition~\ref{prop-quasiopen-char} below. 
On the other hand, it is natural to study Sobolev spaces on quasiopen sets,
as these sets have enough structure to carry reasonable families
of Sobolev functions, and in particular of test functions.
This is important for studying partial differential equations and variational
problems on such sets, see~\cite{KiMa92}.

The metric space approach to Sobolev functions makes it in principle possible
to consider Sobolev spaces and variational problems on arbitrary sets,
but it turns out that there is not much point in considering more general sets
than the quasiopen ones.
More precisely, in  Bj\"orn--Bj\"orn~\cite{BBnonopen} it was shown
that the Dirichlet problem
for \p-harmonic functions in an arbitrary set coincides with the one
in the set's fine interior.
Moreover, the spaces of Sobolev test functions with zero boundary values
are the same for the set and its fine interior.
Even in the metric setting of this paper, 
finely open sets are quasiopen and
quasiopen sets differ from finely open
sets only by sets of capacity zero~\cite{BBL2}.
The results in~\cite{BBnonopen} also show that restrictions
of (upper) gradients from the underlying space (such as $\R^n$) behave
well on quasiopen sets, but not on more general sets.

Quasiopen sets can also be regarded as a link between Sobolev 
spaces and potential theory, 
in which finely open sets play an
important role.
Finely open sets form the fine topology, which is the coarsest topology
making all superharmonic functions continuous~\cite{BBL1}
and serves as a tool for many deep properties in potential theory.
Finely open, and thus quasiopen, sets can be very different from the usual open sets.
The simplest examples of quasiopen sets are open sets with an arbitrary
set of zero capacity removed or added.
Such a removed set can be dense, causing the interior of the resulting
set to be empty. From the point of view of potential theory, such a set behaves like the
original open one. 

A typical finely open set is the complement of the Lebesgue spine in $\R^3$
with the tip of the spine added to it.
This is natural, since for harmonic and superharmonic functions,
the tip behaves more like an interior point than a boundary point.
More generally, finely open sets contain points which have only a ``thin''
connection with the complement of the set. 
Examples~9.5 and~9.6 in~\cite{BBnonopen} describe a closed set $E\subset\R^n$
with positive Lebesgue measure and empty Euclidean interior, but whose fine
interior has full measure in $E$ and is thus suitable for solving the
Dirichlet problem on it.
The set looks like a Swiss cheese and its complement consists of countably
many balls of shrinking radii.
We refer to 
the introductions in \cite{BBnonopen}--\cite{BBL2},
and
the references therein, for discussions on the fine topology and the history of 
the (fine) nonlinear potential theory.

In this paper we assume that $X$ is a complete metric
space equipped with a doubling measure
supporting a \p-Poincar\'e inequality,  $1<p<\infty$.
As hinted before Theorem~\ref{thm-intro-equiv-Sob},
Newtonian functions are better representatives than the usual
Sobolev functions. Namely, it was shown by
Bj\"orn--Bj\"orn--Shan\-mu\-ga\-lin\-gam~\cite{BBS5},
that all Newtonian functions on open sets are quasicontinuous.
Moreover, they are finely continuous outside of sets of zero capacity, by
J.~Bj\"orn~\cite{JB-pfine} or (independently) Korte~\cite{korte08}.
We extend both these results to quasiopen sets in Section~\ref{sect-quasicontinuity}:

\begin{thm} \label{thm-Nploc-intro}
Let $U\subset X$ be quasiopen, and $u \in \Np(U)$. Then $u$ is quasicontinuous in $U$
and finely continuous quasieverywhere in $U$.
\end{thm}

This will be used as an important tool when establishing
Theorem~\ref{thm-intro-equiv-Sob} in Section~\ref{sect-Sobolev}.
Another tool that we will need for proving  both theorems above is the quasi-Lindel\"of principle of the fine topology, which we obtain in
Section~\ref{sect-fine-cont}.
Along the way, in Section~\ref{sect-Newtonian}, 
we introduce  and study several variants of local Newtonian
spaces on quasiopen sets, some of them inspired by the earlier definitions 
in $\R^n$ given by Kilpel\"ainen--Mal\'y~\cite{KiMa92} and 
Mal\'y--Ziemer~\cite{MZ}.
In fact, Theorem~\ref{thm-Nploc-intro} holds for these local
spaces as well, see Theorem~\ref{thm-Nploc}.
They can be used as natural spaces for considering
\p-harmonic and superharmonic functions on quasiopen sets in metric
spaces, as in~\cite{KiMa92}, even though
we will not pursue this line of research here.
These local spaces are also suitable for defining the so-called
\p-fine upper gradients which have been tailored for the proof of
Theorem~\ref{thm-intro-equiv-Sob} but turn out to be a useful generalization
of the minimal \p-weak upper gradients as well,
see Theorems~\ref{thm-minimality} and~\ref{thm-Np-iff-gt-Lp}.

\begin{ack}
The  first two authors were supported by the Swedish Research Council.
Part of this research was done during several visits
of the third author to Link\"oping University in
2012--2015 and
while all three authors
visited Institut Mittag-Leffler in the autumn of 2013.
We thank both institutions for their hospitality
and support.
\end{ack}

\section{Notation and preliminaries}
\label{sect-prelim}

\emph{In Sections~\ref{sect-prelim} and~\ref{sect-quasiopen}
we assume that $1 \le p <\infty$, while in later sections
we will assume that $1<p<\infty$.
Starting from Section~\ref{sect-fine-cont} we will also assume
that $X$ is complete and  supports a \p-Poincar\'e inequality, and that
$\mu$ is doubling, see the definitions below.}

\medskip

We assume throughout the paper
that $X=(X,d,\mu)$ is a metric space equipped
with a metric $d$ and a positive complete  Borel  measure $\mu$
such that $0<\mu(B)<\infty$ for all (open) balls $B \subset X$.
It follows that $X$ is separable,
see Proposition~1.6 in Bj\"orn--Bj\"orn~\cite{BBbook}.
The $\sigma$-algebra on which $\mu$ is defined
is obtained by the completion of the Borel $\sigma$-algebra.
We also assume that $\Om \subset X$ is a nonempty open
set.

We say that $\mu$  is \emph{doubling} if
there exists $C>0$ such that for all balls
$B=B(x_0,r):=\{x\in X: d(x,x_0)<r\}$ in~$X$,
\begin{equation*}
        0 < \mu(2B) \le C \mu(B) < \infty.
\end{equation*}
Here and elsewhere we let $\lambda B=B(x_0,\lambda r)$.
A metric space with a doubling measure is proper\/
\textup{(}i.e.\ closed and bounded subsets are compact\/\textup{)}
if and only if it is complete.

A \emph{curve} is a continuous mapping from an interval,
and a \emph{rectifiable} curve is a curve with finite length.
We will only consider curves which are nonconstant, compact and rectifiable.
A curve can thus be parameterized by its arc length $ds$.
We follow Heinonen and Koskela~\cite{HeKo98} in introducing
upper gradients as follows (they called them very weak gradients).

\begin{deff} \label{deff-ug}
A nonnegative Borel function $g$ on $X$ is an \emph{upper gradient}
of an extended real-valued function $f$
on $X$ if for all nonconstant, compact and rectifiable curves
$\gamma: [0,l_{\gamma}] \to X$,
\begin{equation} \label{ug-cond}
        |f(\gamma(0)) - f(\gamma(l_{\gamma}))| \le \int_{\gamma} g\,ds,
\end{equation}
where we follow the convention that the left-hand side is $\infty$
whenever at least one of the
terms therein is infinite.

If $g$ is a nonnegative measurable function on $X$
and if (\ref{ug-cond}) holds for \p-almost every curve (see below),
then $g$ is a \emph{\p-weak upper gradient} of~$f$.
\end{deff}

Here we say that a property holds for \emph{\p-almost every curve}
if it fails only for a curve family $\Ga$ with zero \p-modulus,
i.e.\ there exists $0\le\rho\in L^p(X)$ such that
$\int_\ga \rho\,ds=\infty$ for every curve $\ga\in\Ga$.
Note that a \p-weak upper gradient need not be a Borel function,
it is only required to be measurable.
On the other hand,
every measurable function $g$ can be modified on a set of measure zero
to obtain a Borel function, from which it follows that
$\int_{\gamma} g\,ds$ is defined (with a value in $[0,\infty]$) for
\p-almost every curve $\ga$.
For proofs of these and all other facts in this section
we refer to Bj\"orn--Bj\"orn~\cite{BBbook} and
Heinonen--Koskela--Shanmugalingam--Tyson~\cite{HKST}.

The \p-weak upper gradients were introduced in
Koskela--MacManus~\cite{KoMc}. It was also shown there
that if $g \in \Lploc(X)$ is a \p-weak upper gradient of $f$,
then one can find a sequence $\{g_j\}_{j=1}^\infty$
of upper gradients of $f$ such that $g_j-g \to 0$ in $L^p(X)$.
If $f$ has an upper gradient in $\Lploc(X)$, then
it has a \emph{minimal \p-weak upper gradient} $g_f \in \Lploc(X)$
in the sense that for every \p-weak upper gradient $g \in \Lploc(X)$ of $f$ we have
$g_f \le g$ a.e., see Shan\-mu\-ga\-lin\-gam~\cite{Sh-harm}.
The minimal \p-weak upper gradient is well defined
up to a set of measure zero in the cone of nonnegative functions in $\Lploc(X)$.
Following Shanmugalingam~\cite{Sh-rev},
we define a version of Sobolev spaces on the metric measure space $X$.

\begin{deff} \label{deff-Np}
Let for measurable $f$,
\[
        \|f\|_{\Np(X)} = \biggl( \int_X |f|^p \, \dmu
                + \inf_g  \int_X g^p \, \dmu \biggr)^{1/p},
\]
where the infimum is taken over all upper gradients of $f$.
The \emph{Newtonian space} on $X$ is
\[
        \Np (X) = \{f: \|f\|_{\Np(X)} <\infty \}.
\]
\end{deff}
\medskip

The space $\Np(X)/{\sim}$, where  $f \sim h$ if and only if $\|f-h\|_{\Np(X)}=0$,
is a Banach space and a lattice, see Shan\-mu\-ga\-lin\-gam~\cite{Sh-rev}.
We also define
\[
   \Dp(X)=\{f : f \text{ is measurable and  has an upper gradient
     in }   L^p(X)\}.
\]
In this paper we assume that functions in $\Np(X)$
and $\Dp(X)$
 are defined everywhere (with values in $\eR:=[-\infty,\infty]$),
not just up to an equivalence class in the corresponding function space.

For a measurable set $E\subset X$, the Newtonian space $\Np(E)$ is defined by
considering $(E,d|_E,\mu|_E)$ as a metric space in its own right.
We say  that $f \in \Nploc(E)$ if
for every $x \in E$ there exists a ball $B_x\ni x$ such that
$f \in \Np(B_x \cap E)$.
The spaces $\Dp(E)$ and $\Dploc(E)$ are defined similarly.
For a function $f \in \Dploc(E)$ we denote the minimal
\p-weak upper gradient of $f$ with respect to $E$ by $g_{f,E}$.

\begin{deff} \label{deff-sobcap}
The \emph{Sobolev capacity} of an arbitrary set $E\subset X$ is
\[
\Cp(E) = \inf_u\|u\|_{\Np(X)}^p,
\]
where the infimum is taken over all $u \in \Np(X)$ such that
$u\geq 1$ on $E$.
\end{deff}

The Sobolev capacity is countably subadditive.
We say that a property holds \emph{quasieverywhere} (q.e.)\
if the set of points  for which the property does not hold
has Sobolev capacity zero.
The Sobolev capacity is the correct gauge
for distinguishing between two Newtonian functions.
If $u \in \Np(X)$, then $u \sim v$ if and only if $u=v$ q.e.
Moreover, Corollary~3.3 in Shan\-mu\-ga\-lin\-gam~\cite{Sh-rev} shows that
if $u,v \in \Np(X)$ and $u= v$ a.e., then $u=v$ q.e.

Capacity is also important for the
following two notions, which are central in this paper.

\begin{deff}   \label{deff-q-cont}
A set $U\subset X$ is \emph{quasiopen} if for every
$\varepsilon>0$ there is an open set $G\subset X$ such that $\Cp(G)<\varepsilon$
and $G\cup U$ is open.

A function $u$ defined on a set  $E \subset X$ is \emph{quasicontinuous}
if for every $\varepsilon>0$ there is an open set $G\subset X$ 
such that $\Cp(G)<\varepsilon$ and $u|_{E \setm G}$ is finite and continuous.
\end{deff}

The quasiopen sets do not in general form a topology,
see Remark~9.1 in Bj\"orn--Bj\"orn~\cite{BBnonopen}.
However it follows easily from the countable subadditivity of $\Cp$
that countable unions and finite intersections of quasiopen sets are quasiopen.
(We consider finite sets to be countable
throughout the paper.)
For characterizations of quasiopen sets and quasicontinuous functions
see  Bj\"orn--Bj\"orn--Mal\'y~\cite{BBMaly} and 
Proposition~\ref{prop-quasiopen-char}.

Together with the doubling property defined above, the following
Poincar\'e inequality is often a standard assumption on metric spaces.

\begin{deff} \label{def.PI.}
We say that $X$ supports a \emph{\p-Poincar\'e inequality} if
there exist constants $C>0$ and $\lambda \ge 1$
such that for all balls $B \subset X$,
all integrable functions $f$ on $X$ and all upper gradients $g$ of $f$,
\begin{equation} \label{PI-ineq}
        \vint_{B} |f-f_B| \,\dmu
        \le C \diam(B) \biggl( \vint_{\lambda B} g^{p} \,\dmu \biggr)^{1/p},
\end{equation}
where $ f_B
 :=\vint_B f \,\dmu
:= \int_B f\, d\mu/\mu(B)$.
\end{deff}

In the definition of Poincar\'e inequality we can equivalently assume
that $g$ is a \p-weak upper gradient.
In $\R^n$ equipped with a doubling measure $d\mu=w\,dx$, where
$dx$ denotes Lebesgue measure, the \p-Poincar\'e inequality~\eqref{PI-ineq}
is equivalent to the \emph{\p-admissibility} of the weight $w$ in the
sense of Heinonen--Kilpel\"ainen--Martio~\cite{HeKiMa}, see
Corollary~20.9 in~\cite{HeKiMa}
and Proposition~A.17 in~\cite{BBbook}.

If $X$ is complete and  supports a \p-Poincar\'e inequality
and $\mu$ is doubling, then Lipschitz functions
are dense in $\Np(X)$, see Shan\-mu\-ga\-lin\-gam~\cite{Sh-rev}.
Moreover, all functions in $\Np(X)$
and those in $\Np(\Om)$ are quasicontinuous,
see Bj\"orn--Bj\"orn--Shan\-mu\-ga\-lin\-gam~\cite{BBS5}.
This means that in the Euclidean setting, $\Np(\R^n)$ is the
refined Sobolev space as defined in
Heinonen--Kilpel\"ainen--Martio~\cite[p.~96]{HeKiMa},
see \cite[Appendix~A.2]{BBbook}
for a proof of this fact valid in weighted $\R^n$.
This is the main reason why, unlike in the classical Euclidean setting,
we do not need to
require the functions competing in the
definitions of capacity to be $1$ in a
neighbourhood of $E$.
For recent related progress on the density of Lipschitz functions
see Ambrosio--Colombo--Di Marino~\cite{AmbCD} and
Ambrosio--Gigli--Savar\'e~\cite{AmbGS}.

In Section~\ref{sect-fine-cont} the fine topology is
defined by means of thin sets, which in turn use the
variational capacity $\cp$.
To be able to define the variational capacity we first
need a Newtonian space with zero boundary values.
We let, for an arbitrary set $A \subset X$,
\[
\Np_0(A)=\{f|_{A} : f \in \Np(X) \text{ and }
        f=0 \text{ on } X \setm A\}.
\]
One can replace the assumption ``$f=0$ on $X \setm A$''
with ``$f=0$ q.e.\ on $X \setm A$''
without changing the obtained space
$\Np_0(A)$.
Functions from $\Np_0(A)$ can be extended by zero in $X\setm A$ and we
will regard them in that sense if needed.

\begin{deff}
The \emph{variational
capacity} of $E\subset \Om$ with respect to $\Om$ is
\[
\cp(E,\Om) = \inf_u\int_X g_u^p\, d\mu,
\]
where the infimum is taken over all $u \in \Np_0(\Om)$
such that
$u\geq 1$ on $E$.
\end{deff}

\section{Quasiopen and \texorpdfstring{\p}{p}-path open sets}

\label{sect-quasiopen}

In this section we assume that $1 \le p <\infty$,
but $\mu$ is not required to be doubling
and, apart from Proposition~\ref{prop-Dp=infty}, no Poincar\'e inequality
is required.

Recall that quasiopen sets were defined in Definition~\ref{deff-q-cont}.
The following is another useful notion in connection with Sobolev
functions. 
It was introduced by Shanmugalingam~\cite{Sh-harm}.

\begin{deff}
A set $U \subset X$ is \emph{\p-path open} if for
\p-almost every curve $\ga:[0,l_\ga]\to X$, the set
$\ga^{-1}(U)$ is (relatively) open in $[0,l_\ga]$.
\end{deff}

\begin{lem} \label{lem-quasiopen-ppath-meas}
\textup{(Shanmugalingam~\cite[Remark~3.5]{Sh-harm} and
Bj\"orn--Bj\"orn~\cite[Lemma~9.3]{BBnonopen})}
If $U$ is quasiopen, then $U$ is \p-path open and measurable.
\end{lem}

The following lemma is from \cite[Corollary~3.7]{BBnonopen}.
Recall that $g_{u,U}$ denotes the minimal \p-weak upper gradient
of $u$ with respect to $U$, while $g_u$ denotes the
minimal \p-weak upper gradient
of $u$ with respect to $X$.

\begin{lem} \label{lem-minimal-restrict}
Let $U$ be a measurable \p-path open set
and $u \in \Dploc(X)$.
Then
\[
    g_{u,U}=g_u
    \quad \text{a.e.\ in } U.
\]
In particular, this holds if $U$ is quasiopen.
\end{lem}

For a  family of curves $\Gamma$ on $X$,
we define its \emph{\p-modulus}
\[
         \Modp(\Gamma):=\inf \int_X \rho^p \, d\mu,
\]
where the infimum is taken over all nonnegative Borel functions $\rho$
such that $\int_\ga \rho \,ds \ge 1$ for all $\ga \in \Gamma$.
Let $\Ga_E^U$ be the set of curves $\ga:[0,l_\ga]\to U$  which hit $E \subset U$, i.e.\  $\ga^{-1}(E) \ne \emptyset$.

We are now ready to make some new observations
for \p-path open sets.

\begin{lem} \label{lem-Modp}
If $U$ is \p-path open and $E \subset U$,
then
$\Modp(\Ga_E^U)=\Modp(\Ga_E^X)$.
\end{lem}

If $U$ is not \p-path open, then this is not true in general:
Consider, e.g., $E=\{0\} \subset \R$ and $U=\Q$, in which case
$\Modp(\Ga_E^U)=0 < \Modp(\Ga_E^\R)$. 

\begin{proof}
Since $\Ga_E^U \subset \Ga_E^X$, we have
$
    \Modp(\Ga_E^U) \le \Modp(\Ga_E^X)
$.

Conversely, as $U$ is \p-path open, \p-almost every curve $\ga \in \Ga_E^X$
is such that $\ga^{-1}(U)$ is relatively open in $[0,l_\ga]$
(and we can ignore the other curves  in $\Ga_E^X$).
Moreover $\ga^{-1}(U) \supset \ga^{-1}(E) \ne \emptyset$.
Hence $\ga^{-1}(U)$ is a nonempty
countable union of relatively open
intervals of $[0,l_\ga]$ (see e.g.\ Lemma~1.4 in \cite{BBbook}), at
least one of which contains a point $t \in \ga^{-1}(E)$.
We can thus find a small compact interval $[a,b] \ni t$, $0 \le a <b \le l_\ga$,
such that $[a,b] \subset \ga^{-1}(U)$.
Then $\ga|_{[a,b]} \in \Ga_E^U$,
and Lemma~1.34\,(c)  in \cite{BBbook} implies that 
$\Modp(\Ga_E^X) \le \Modp(\Ga_E^U)$.
\end{proof}

\begin{prop} \label{prop-Dp=infty}
Assume that $X$ supports a \p-Poincar\'e inequality.
Let $U$ be a measurable \p-path open set and
$u \in \Dp(U)$.
Then  $\Cp(\{x \in U : |u(x)|=\infty\})=0$.
\end{prop}

\begin{proof}
Let $E=\{x \in U : |u(x)|=\infty\}$.
As $u \in \Dp(U)$, $u$ is measurable and thus also $E$ is measurable.
Let $g \in L^p(U)$ be an upper gradient of $u$ in $U$.
Each curve $\ga \in \Ga_E^U$
contains a nonconstant subcurve $\gat$ that either starts
or ends in $E$.
As $g$ is an upper gradient of $u$ in $U$ and $|u(x)|=\infty$ for $x \in E$,
we see that
\[
     \int_{\ga} g\,ds \ge \int_{\gat} g\,ds=\infty.
\]
Hence,
$\Modp(\Ga_E^U)=0$.
By Lemma~\ref{lem-Modp}, $\Modp(\Ga_E^X)=\Modp(\Ga_E^U)=0$.
Finally, Proposition~4.9 in \cite{BBbook} shows that
$\Cp(E)=0$.
\end{proof}

\section{Fine topology}
\label{sect-fine-cont}

\emph{Throughout the rest of the paper, we assume that
$1<p<\infty$, that $X$ is complete and  supports a \p-Poincar\'e inequality, and that
$\mu$ is doubling.}

\medskip

In this section we recall the basic facts about the fine topology associated
with Sobolev spaces and prove some auxiliary results which will be crucial in the
subsequent sections.

\begin{deff}\label{deff-thinness}
A set $E\subset X$ is  \emph{thin} at $x\in X$ if
\begin{equation}   \label{deff-thin}
\int_0^1\biggl(\frac{\cp(E\cap B(x,r),B(x,2r))}{\cp(B(x,r),B(x,2r))}\biggr)^{1/(p-1)}
     \frac{dr}{r}<\infty.
\end{equation}
A set $U\subset X$ is \emph{finely open} if
$X\setminus U$ is thin at each point $x\in U$.
\end{deff}

It is easy to see that the finely open sets give rise to a
topology, which is called the \emph{fine topology}.
Every open set is finely open, but the converse is not true in general.

In the definition of thinness,
we make the convention that the integrand
is 1 whenever $\cp(B(x,r),B(x,2r))=0$.
This happens e.g.\ if $X=B(x,2r)$,
but never
if $r < \frac{1}{2}\diam X$.
Note that thinness is a local property.

\begin{deff}
A function $u : U \to \eR$, defined on a finely open set $U$, is
\emph{finely continuous} if it is continuous when $U$ is equipped with the
fine topology and $\eR$ with the usual topology.
\end{deff}

Since every open set is finely open,
the fine topology
generated by the
finely open\index{finely!open} sets is finer than the metric topology.
In fact, it is
the coarsest topology making all superharmonic functions
finely continuous, by
J.~Bj\"orn~\cite[Theorem~4.4]{JB-pfine},
Korte~\cite[Theorem~4.3]{korte08}
and Bj\"orn--Bj\"orn--Latvala~\cite[Theorem~1.1]{BBL1}.
See \cite[Section~11.6]{BBbook} and \cite{BBL1}
for further discussion on thinness and the fine topology.

The following definition will play an important role
in the later sections of the paper.

\begin{deff}
A set $E \subset A$ is a \emph{\p-strict subset} of $A$ if
there is a function $u \in \Np_0(A)$ such that $u =1$ on $E$.
\end{deff}

Equivalently, it can in addition be required
that $0 \le u \le 1$,
as in Kilpel\"ainen--Mal\'y~\cite{KiMa92}.
The following lemma shows that there are many nice \p-strict subsets
of finely open sets.
They play the role of compact subsets of open sets.
In particular, there is a base of fine neighbourhoods, consisting only of
\p-strict subsets.

\begin{lem} \label{lem-finely-open-x}
Let $V$ be finely open and let $x_0\in V$.
Then there exist a finely open $W\ni x_0$ and
an upper semicontinuous finely continuous function $v\in \Np_0(V)$
with compact support in $V$ such that $0\le v\le1$ everywhere
and $v=1$ in $W$.

In particular, $W$ is a \p-strict subset of $V$ and $W\Subset V$.
\end{lem}

Recall that $W\Subset V$
if $\overline{W}$ is a compact subset of $V$.

\begin{proof}
Since $V$ is finely  open, $E:=X \setm V$ is thin at $x_0$.
By the weak Cartan property
in Bj\"orn--Bj\"orn--Latvala~\cite[Theorem~5.1]{BBL1}
there exists a lower semicontinuous finely continuous function $u\in\Np(B)$
in a ball $B\ni x_0$ such that $0<u\le 1$ in $B$, $u=1$ in
$E\cap B$ and $u(x_0)<1$.
Let $0\le\eta\le1$ be a Lipschitz function with compact support in $B$
such that $\eta(x_0)=1$.
Then $f=\eta(1-u)\in\Np_0(V)$ is upper semicontinuous and finely continuous
in $X$ and $f(x_0)=1-u(x_0)>0$.

To conclude the proof, set
\[
W=\bigl\{x\in V: f(x)>\tfrac12 f(x_0)\bigr\}
\quad \text{and} \quad
v(x) = \min\biggl\{ 1, \biggl( \frac{4f(x)}{f(x_0)}-1\biggr)_\limplus \biggr\}.
\]
A simple calculation shows that $v=1$ in $W$,
and the upper semicontinuity of $f$ implies that
$W\Subset \spt v \Subset V$.
As $f$ is finely continuous, $W$ is finely open.
\end{proof}

The following quasi-Lindel\"of principle
will play an important role in the later sections.

\begin{thm} \label{thm-quasiLindelof}
\textup{(Quasi-Lindel\"of principle)}
For each family $\mathcal V$ of finely open sets there is a countable subfamily $\mathcal V'$ such that
\[
\Cp\biggl(\bigcup_{V\in\mathcal V}V\setminus \bigcup_{V'\in\mathcal V'}V'\biggr)=0.
\]
\end{thm}

Our proof of the quasi-Lindel\"of principle
in metric spaces is quite similar to the proof
in Heinonen--Kilpel\"ainen--Mal\'y~\cite[Theorem~2.3]{HeKiMaly}
for unweighted $\R^n$. We include it here for the reader's convenience.
We will need the following lemma.

\begin{lemma}\label{lem:wienersumapprox}
Let $x\in X$ and let $\{x_k\}_{k=1}^\infty$ be a sequence of points in $X$
such that $d(x_k,x)<2^{-k-2}$ for $k=1,2,\dots$\,.
If $B_k=B(x_k,3\cdot 2^{-k-2})$,
then a set $E\subset X$ is thin at $x$ if and only if
\[
\sum_{k=1}^{\infty}\biggl(
   \frac{\cp(E\cap B_k, 2B_k)}{\cp(B_k, 2B_k)}
\biggr)^{1/(p-1)}<\infty.
\]
\end{lemma}

\begin{proof}
Since $B(x,2^{-k-1})\subset B_k\subset B(x,2^{-k})$,
the claim is an easy consequence of
Lemma~4.6  in Bj\"orn--Bj\"orn--Latvala~\cite{BBL1}
and Lemma~3.3 in~\cite{Bj} (or Proposition~6.16 in~\cite{BBbook}).
\end{proof}

\begin{proof}[Proof of Theorem~\ref{thm-quasiLindelof}]
Since $X$ is separable there is a countable dense subset $Z \subset X$.
Let $\{B_k\}_{k=1}^\infty$ be
an enumeration
of all open balls with rational radii $< \tfrac{1}{4}\diam X$
and  centres in $Z$.
We define a monotone set function $\Phi$ by setting
\[
\Phi(E)=\sum_{k=1}^{\infty}2^{-k}\frac{\cp(E\cap B_k, 2B_k)}{\cp(B_k, 2B_k)}
\]
for any $E\subset X$.
Note that $\Phi(E)\le1$, and
$\Cp(E)=0$ if and only if $\Phi(E)=0$.

Let $U:=\bigcup_{V\in\mathcal V}V$ and define
\[
\delta:=\inf\biggl\{\Phi\biggl(U\setminus\bigcup_{W\in\mathcal W}W\biggr):
     \mathcal W\subset \mathcal V\ \text{is countable}\biggr\}
 \le 1.
\]
Then for each $j=1,2,\dots$ we may choose a countable subfamily $\mathcal V_j'$ of $\mathcal V$ such that
\[
\Phi\biggl(U\setminus\bigcup_{V'\in\mathcal V_j'}V'\biggr)\le \delta +\frac1{j}.
\]
By defining $\mathcal V':=\bigcup_{j=1}^{\infty}\mathcal V_j'$ we then have $\delta=\Phi(F)$ for
$F=U \setminus\bigcup_{V'\in\mathcal V'}V'$.

It is enough to show that $\delta=0$. Assume on the contrary that $\delta>0$.
We now invoke the fine Kellogg property from
Bj\"orn--Bj\"orn--Latvala~\cite[Corollary~1.3]{BBL2}
which says that
\[
C_p(E\setminus b_pE)=0
\quad \text{for any set } E \subset X,
\]
where the \emph{base} $b_p E$ consists
of all points in $X$ at which $E$ is not thin.
Since $\Cp(F)>0$, this implies that
\[
0<\Cp(F)\le\Cp(F\cap b_p F)+\Cp(F\setminus b_p F)=\Cp(F\cap b_p F).
\]

Accordingly, there exists $x\in F\cap b_pF\subset U$. Choose $V_0\in\mathcal V$ such that $x\in V_0$.
Then the set $F\setminus V_0$ is thin at $x$ whereas $F$ is not thin at $x$.
Now Lemma~\ref{lem:wienersumapprox} implies the existence of a ball
$B_k$ with centre in $Z$
and rational radius $< \frac{1}{4} \diam X$  such that
\[
\cp((F\setminus V_0)\cap B_k,2B_k)<\cp(F\cap B_k,2B_k).
\]
Then we have $\Phi(F\setminus V_0)<\Phi(F)=\delta$, which is a contradiction. Hence the claim holds.
\end{proof}

\begin{remark}
In the proof  of the quasi-Lindel\"of principle
we used the fine Kellogg property, whose proof
depends in turn on the Choquet property, the Cartan property
and ultimately on the Cheeger differentiable
structure, see Bj\"orn--Bj\"orn--Latvala~\cite{BBL2}.
It would be nice if one could obtain a more elementary proof
of the fine Kellogg property.

In fact, here we only used a seemingly milder
consequence of
the fine Kellogg property: if $\Cp(F)>0$, then $F \cap b_p F \ne \emptyset$.
However this consequence is equivalent to the full fine Kellogg property,
which can be seen as follows:
Let $F=E \setm b_p E$ for some set $E$.
As $F \subset E$, we directly have $b_p F \subset b_p E$, and thus
$F \cap b_p F \subset (E \setm b_p E) \cap b_p E = \emptyset$.
Hence the assumption implies that $\Cp(E \setm b_p E)=\Cp(F)=0$.
\end{remark}

We can now characterize quasiopen sets in several ways.
See Kilpel\"ainen--Mal\'y~\cite[Theorem~1.5]{KiMa92}
for the corresponding result in $\R^n$.

\begin{prop} \label{prop-quasiopen-char}
Let $U\subset X$ be arbitrary.
Then the following conditions are equivalent\/\textup{:}
\begin{enumerate}
\renewcommand{\theenumi}{\textup{(\roman{enumi})}}%
\item \label{b-1}
$U$ is quasiopen\/\textup{;}
\item \label{b-2}
$U$ is a union of a finely open set and a set of capacity zero\/\textup{;}
\item \label{b-3}
$U=\{x:u(x)>0\}$ for some nonnegative quasicontinuous $u$ on
$X$\textup{;}
\item \label{b-4}
$U=\{x:u(x)>0\}$ for some nonnegative $u\in \Np(X)$.
\end{enumerate}
\end{prop}

\begin{proof}
\ref{b-1} \imp \ref{b-2} This is Theorem~4.9\,(a) in
Bj\"orn--Bj\"orn--Latvala~\cite{BBL1}.

\ref{b-2} \imp \ref{b-1}
This follows from Theorem~1.4\,(a) in Bj\"orn--Bj\"orn--Latvala~\cite{BBL2}.

\ref{b-4} \imp \ref{b-3}
By Theorem~1.1 in Bj\"orn--Bj\"orn--Shan\-mu\-ga\-lin\-gam~\cite{BBS5},
$u$ is quasicontinuous and thus \ref{b-3} holds.

\ref{b-3} \imp \ref{b-1} This is well known and easy to prove.

\ref{b-2} \imp \ref{b-4}
Assume that $V\subset U$ is finely open and $\Cp(U\setminus V)=0$.
By Lemma~\ref{lem-finely-open-x}, for each $x\in V$
we find $v_x\in \Np_0(V)$ such that $0\le v_x\le 1$, $v_x(x)=1$
and $v_x=0$ outside of $V$.
Since $v_x$ is quasicontinuous, $\{y:v_x(y)>0\}$ is
a quasiopen subset of $V$.
Therefore, by the quasi-Lindel\"of principle (Theorem~\ref{thm-quasiLindelof}),
the family $\{v_x:x\in V\}$ contains a
countable subfamily $\{v_j\}_{j=1}^\infty$
such that $\Cp(Z)=0$ for the set
\[
Z:=\{x\in V: v_j(x)=0 \textrm{ for all } j\}.
\]
Choose $a_j>0$ so that
$v:=\sum_{j=1}^\infty a_jv_j \in \Np(X)$
and let
\[
u=\begin{cases}
  1 & \textrm{on }  Z\cup(U\setminus V),\\
  v & \textrm{elsewhere}.
\end{cases}
\]
Since $u=v$ q.e., also $u \in \Np(X)$.
Moreover $U=\{x:u(x)>0\}$.
\end{proof}

\section{Local Newtonian spaces on quasiopen sets}
\label{sect-Newtonian}

Solutions of differential equations and variational problems
are usually considered in local Sobolev spaces.
On quasiopen sets there
are at least three different natural candidates for a
local Newtonian space. 
The first is $\Nploc(U)$ that we introduced already in Section~\ref{sect-prelim}.
The others are $\Npploc(U)$ and $\Npqloc(U)$ that we introduce below.

The space $\Nploc(U)$ is natural in connection with Newtonian spaces on metric spaces,
but for the classical Sobolev spaces
it is a less obvious definition to consider.
On $\R^n$ it is not even obvious how to define (nonlocal) Sobolev spaces on quasiopen
sets. A fruitful definition of $\Wp(U)$ was given by
Kilpel\"ainen--Mal\'y~\cite{KiMa92} and they also introduced a
local space (called $\Wploc(U)$ therein)
 which is similar to our definition of $\Npploc(U)$.
The same definitions are used in Latvala~\cite{LatPhD}.
Yet another definition of a local space was considered
by Mal\'y--Ziemer~\cite[p.~149]{MZ} (and called
$W^{1,p}_{\textup{\p-loc}}(U)$), see also Open problem~\ref{openprob-p-loc}.
This last definition inspired our definition of $\Npqloc(U)$ below.
To give it we first need to discuss quasicoverings.

\begin{deff}
A family $\B$ of quasiopen sets is a \emph{quasicovering} of a set $E$ if
it is countable, $\bigcup_{U \in \B} U\subset E$
and $\Cp(E\setm \bigcup_{U \in \B} U) =0$.
If every $V \in \B$ is a finely open \p-strict subset of $E$
and $\overline{V} \Subset E$, then
$\B$ is a \emph{\p-strict quasicovering} of $E$.
\end{deff}

This definition of a quasicovering is slightly different from
the one in Kilpel\"ainen--Mal\'y~\cite{KiMa92}, in that we require
the covering to be countable and that it consists of subsets of $E$.
This will be convenient for us, as we have no use for other quasicoverings
in this paper.
Moreover with our definition we get another
characterization of quasiopen sets in Lemma~\ref{lem-exist-quasicover} below.

\begin{prop}   \label{prop-strict-quasicover}
Let $U$ be quasiopen.
Then there exists a \p-strict quasicovering $\B$ of $U$ consisting of finely
open sets. Moreover, the Newtonian functions  associated with the \p-strict
subsets in $\B$ can be chosen compactly supported in $U$.
\end{prop}

\begin{proof} By Proposition~\ref{prop-quasiopen-char}
we can write $U=V \cup E$, where $V$ is finely open and $\Cp(E)=0$.
For every $x\in V$,
Lemma~\ref{lem-finely-open-x}
provides us with a finely open
set $V_x\ni x$ such that $V_x\Subset V$
and the corresponding function $v_x\in\Np_0(V)$
has compact support in $V$.
The collection $\B'=\{V_x\}_{x\in V}$ covers $V$
and by the quasi-Lindel\"of principle (Theorem~\ref{thm-quasiLindelof}),
and the fact that $\Cp(E)=0$,
there exists a countable subcollection $\B \subset \B'$ such that
$\Cp(U\setm\bigcup_{V_x\in\B}V_x)=0$.
\end{proof}

\begin{lem} \label{lem-exist-quasicover}
The set $E$ has a quasicovering if and only if it  is quasiopen.
\end{lem}

\begin{proof}
If $E$ is quasiopen, then obviously $\{E\}$ is a quasicovering of $E$.

Conversely, if $E$ has a quasicovering $\{U_j\}_{j=1}^\infty$, then
$U=\bigcup_{j=1}^\infty U_j$ is quasiopen, by the countable
subadditivity of $\Cp$.
As $\Cp(E \setm U)=0$ and
$\Cp$ is an outer capacity,
by Corollary~1.3 in Bj\"orn--Bj\"orn--Shan\-mu\-ga\-lin\-gam~\cite{BBS5}
(or Theorem~5.31 in \cite{BBbook}),
also $E$ is quasiopen.
\end{proof}

We now turn to the different local Newtonian and Dirichlet spaces.
Recall that $\Nploc(A)$
and $\Dploc(A)$ were already defined in Section~\ref{sect-prelim}.

\begin{deff} \label{deff-Npploc}
Let $A$ be measurable.
We  say that

\begin{enumerate}
\renewcommand{\theenumi}{\textup{(\roman{enumi})}}%
\item
$u \in \Npploc(A)$
if $u \in \Np(V)$
for all finely open \p-strict subsets $V \Subset A$\textup{;}
\item
 $u \in \Npqloc(A)$
if
there is a quasicovering $\B$ of $A$ such that
$u \in \Np(U)$
for all $U \in \B$.
\end{enumerate}

\vskip-\bigskipamount 

We similarly define $\Dpploc(A)$, $\Dpqloc(A)$, $\Lpploc(A)$ and $\Lpqloc(A)$.
\end{deff}

Note that in view of Lemma~\ref{lem-exist-quasicover}
the definitions of $\Npqloc(A)$ and  $\Dpqloc(A)$ are
useful only when $A$ is quasiopen, otherwise
$\Npqloc(A)= \Dpqloc(A) = \emptyset$.

Let us take a moment to
study these spaces and compare them to $\Nploc(A)$ and $\Dploc(A)$.
The following lemma belongs to folklore, but as it is used several times,
we state it here and include the short proof.
Together with Proposition~\ref{prop-quasiopen-char}
it implies that in the definition of $\Npploc(A)$ (or $\Dpploc(A)$)
one can equivalently require that
$u \in \Np(U)$ (or $u \in \Dp(U)$)
for all quasiopen \p-strict subsets $U \Subset A$.
It can also be seen as a motivation for the definitions
of $\Npqloc(A)$  and $\Dpqloc(A)$.

\begin{lem} \label{lem-quasiopen-char}
Let $E\subset A$ be measurable sets such that $\Cp(A\setm E)=0$.
If $u:A\to\eR$, then $g:A\to[0,\infty]$
is a \p-weak upper gradient of $u$ with respect to $A$
if and only if it is a \p-weak upper gradient with respect to $E$.
In particular, $\Np(A)=\Np(E)$ with equal norms and equal minimal \p-weak upper
gradients.

Similar statements hold for $\Dp$, $\Npqloc$  and $\Dpqloc$.
\end{lem}

Note that the corresponding statements for
$\Nploc$, $\Dploc$, $\Npploc$ and $\Dpploc$ are false,
as seen by e.g.\ letting $A$ be the unit ball in $\R^n$ with $n \ge p$
and $E=A \setm \{0\}$.

\begin{proof}
Clearly, $\Np(A)\subset\Np(E)$ and every \p-weak upper gradient with respect
to $A$ is a \p-weak upper gradient with respect to $E$.

Conversely, let $g$ be a \p-weak upper gradient of $u$ in $E$.
Proposition~1.48 in \cite{BBbook} shows that
$\mu(A \setm E)=0$ and that \p-almost every curve in $A$ avoids $A \setm E$.
From the latter
we immediately see that $g$ is also a \p-weak upper gradient
in $A$.
As $\mu(A \setm E)=0$,  the equality
of the $\Np$-spaces and their norms follows.

Since a quasicovering of $E$ is automatically a quasicovering of $A$,
we immediately have that $\Npqloc(E) \subset \Npqloc(A)$.
Conversely, if $u \in \Npqloc(A)$ then there is a quasicovering
$\{U_j\}_{j=1}^\infty$ of $A$ such that $u \in \Np(U_j)$ for each $j$.
As $\Cp(U_j \setm E)=0$, we see that $U_j \cap E$ is quasiopen,
and thus $\{U_j \cap E\}_{j=1}^\infty$
is a quasicovering of $E$.
Since $u \in \Np(U_j \cap E)$ for all $j$,
we obtain that $u \in \Npqloc(E)$.

The proofs for $\Dp$  and $\Dpqloc$ are similar.
\end{proof}

\begin{prop} \label{prop-Nploc=>Npploc}
\begin{enumerate}
\renewcommand{\theenumi}{\textup{(\roman{enumi})}}%
\item \label{j-meas}
If $A$ is  measurable, then
$\Nploc(A) \subset \Npploc(A)$.
\item \label{j-quasiopen}
If $U$ is quasiopen, then $\Nploc(U) \subset \Npploc(U) \subset \Npqloc(U)$.
\item \label{j-open}
If $\Om$ is  open, then $\Nploc(\Om)= \Npploc(\Om)  \subset \Npqloc(\Om)$.
\item \label{j-open-qopen}
If $\Om$ is  open and all quasiopen subsets of $\Om$ are open,
then
\begin{equation} \label{Nploc=Npqloc}
   \Nploc(\Om)= \Npploc(\Om)  = \Npqloc(\Om).
\end{equation}
\end{enumerate}
Similar statements hold for $\Dp$. Moreover,
if $\Om$ is open,
then
\begin{equation} \label{Nploc=Dploc}
   \Nploc(\Om)= \Npploc(\Om)=\Dploc(\Om)= \Dpploc(\Om).
\end{equation}
\end{prop}

Note that in general $\Nploc(\Om) \ne \Npqloc(\Om)$ even when $\Om$ is
open. Let e.g.\ $\Om$ be the unit ball $B$ in $\R^n$ with
$n \ge p$. Then any function $u \in \Nploc(B \setm \{0\})$ belongs to
$\Npqloc(B)$, but not all such functions belong to $\Nploc(B)$.

\begin{openprob}
Is the first inclusion in \ref{j-quasiopen} sometimes strict?
\end{openprob}

\begin{openprob} \label{openprob-DplocU}
Is $\Nploc(U)=\Dploc(U)$
and/or $\Npploc(U)=\Dpploc(U)$
if $U$ is quasiopen?
\end{openprob}

We will later show that
$\Npqloc(U)=\Dpqloc(U)$ if $U$ is quasiopen,
see Corollary~\ref{cor-Npqloc=Dpqloc}.
Our proof of this equality is quite involved, and it would be interesting
to know if it can be deduced more easily.
(If $U$ is not quasiopen, then $\Npqloc(U)=\Dpqloc(U)=\emptyset$.)

\begin{openprob}
Another open problem, related to 
Proposition~\ref{prop-Nploc=>Npploc}\,\ref{j-open-qopen}, is whether
all quasiopen sets are open if and only if
the capacity of every point is positive.
One implication is clear: if $\Cp(\{x\})=0$, then $\{x\}$ is a quasiopen nonopen
set, see the proof below.
\end{openprob}

A positive answer to this last question would
lead to a positive answer to Open problem~5.38
in \cite{BBbook}.

\begin{proof}[Proof of Proposition~\ref{prop-Nploc=>Npploc}]
\ref{j-meas}
 Let $f \in \Nploc(A)$ and let $V \Subset A$ be a
finely open \p-strict subset.
For every $x \in \overline{V}$ there is $r_x>0$ such that $f \in \Np(B(x,r_x) \cap A)$.
Since $\overline{V}$ is compact there is a finite subcover
$\{B(x_j,r_{x_j}) \cap A\}_{j=1}^m$ such that
$\overline{V} \subset \bigcup_{j=1}^m B(x_j,r_{x_j}) \cap A$.
It follows from Lemma~\ref{lem-minimal-restrict}
that $\|f\|_{\Np(\overline{V})}^p \le \sum_{j=1}^m
      \|f\|_{\Np(B(x_j,r_{x_j}) \cap A)}^p < \infty$, and thus
     $f \in \Np(V)$.

\ref{j-quasiopen}
The first inclusion was established in  \ref{j-meas},
so assume that $f \in \Npploc(U)$.
By Proposition~\ref{prop-strict-quasicover}, $U$ has a
\p-strict quasicovering $\{V_j\}_{j=1}^\infty$ consisting of
finely open \p-strict subsets $V_j$ such that $\overline{V}_j \Subset U$.
By assumption, $f \in \Np(V_j)$ for each $j$, and thus $f \in \Npqloc(U)$.

\ref{j-open}
 Let $f \in \Npploc(\Om)$ and $x \in \Om$.
Then there is $r_x$ such that $B(x,r_x) \Subset \Om$.
It is straightforward to see that $B(x,r_x)$ is a \p-strict
subset of $\Om$, and thus
$f \in \Np(B(x,r_x))$.
Hence $f \in \Nploc(\Om)$.
The remaining inclusions follow
from \ref{j-quasiopen}.

\ref{j-open-qopen}
Let $f \in \Npqloc(\Om)$.
Then there is a quasicovering $\{U_j\}_{j=1}^\infty$ of $\Om$
so that $f \in \Np(U_j)$.
We shall
show that $\Om =G:=\bigcup_{j=1}^\infty U_j$.
Assume on the contrary that there is
a point $ x\in \Om \setm G$.
As $\Cp(\{x\}) \le \Cp(\Om \setm G)=0$
and $\Cp$ is an outer capacity
(by Bj\"orn--Bj\"orn--Shan\-mu\-ga\-lin\-gam~\cite[Corollary~1.3]{BBS5}
or \cite[Theorem~5.31]{BBbook}),
$\{x\}$ is quasiopen.
But since $X$ is connected (which follows from the Poincar\'e inequality,
see Proposition~4.2 in \cite{BBbook}),
$\{x\}$ is not open, a contradiction.
Hence $\Om=G$, and thus for each $x \in \Om$ there is $j$ such
that $x \in U_j$ and $f \in \Np(U_j)$.
By assumption $U_j$ is open, and hence
$f \in \Nploc(\Om)$.
The other inclusions in
\eqref{Nploc=Npqloc} follow from \ref{j-open}.

The corresponding results for $\Dp$ are proved similarly.
Finally, if $\Om$ is open, then
$\Nploc(\Om)=\Dploc(\Om)$, by
Proposition~4.14 and Corollary~4.24 in \cite{BBbook}, and hence
 \eqref{Nploc=Dploc} follows from \ref{j-open} and the corresponding
result for $\Dp$.
\end{proof}

\begin{example}
One may ask if the requirement $V \Subset A$ in the definition of $\Npploc(A)$
is essential, or stated in another way:
If $u \in \Npploc(A)$ and $V$ is a finely open \p-strict subset of $A$,
does it then follow that $u \in \Np(V)$?
This may fail even for open $A=\Om$ and open \p-strict subsets $V$ of $\Om$.
Indeed, let
$1<p<2$,
\[
   \Om=(-1,1) \times (0,1) \subset \R^2
   \quad \text{and} \quad
  V=\bigl\{(x_1,x_2) \in \Om : |x_1| < x_2 < \tfrac{1}{6}\bigr\}.
\]
Then $V$ is an open \p-strict subset of $\Om$,
by Example~5.7 in Bj\"orn--Bj\"orn~\cite{BB}
(or Example~11.10 in \cite{BBbook})
as $f$ therein belongs to $\Np_0(\Om)$ and $f=1$ on $V$.
Let $u(x)=|x|^\be$ with $\be \le -2/p$.
Then, $u \in \Nploc(\Om)=\Npploc(\Om)$
(cf.\  Proposition~\ref{prop-Nploc=>Npploc}\,\ref{j-open}),
but $u \notin L^p(V)$ and thus $u \notin \Np(V)$.
\end{example}

On a finely open set $V$ yet another local Newtonian space may be considered:
$u \in \Npfloc(V)$ if for every $x \in V$ there is a finely open set
$V_x \subset V$ such that $x \in V_x$ and $u \in \Np(V_x)$.
By Lemma~\ref{lem-finely-open-x}, one can equivalently require that
$V_x \Subset V$.
Lemma~\ref{lem-finely-open-x} and the
quasi-Lindel\"of principle (Theorem~\ref{thm-quasiLindelof})
immediately imply that
\begin{equation} \label{eq-Npfloc}
   \Npploc(V) \subset \Npfloc(V) \subset \Npqloc(V).
\end{equation}
It is natural to ask if these inclusions can be strict.
Indeed this is so even if $V$ is open,
as we show in the following examples.

\begin{example}
To see that the first inclusion in \eqref{eq-Npfloc} can be strict,
let $V=B$ be the unit ball in $\R^3$, $p=2$ and let
\[
   E=\{(x,t) \in \R^2 \times \R : t>0 \text{ and }
       |x| <  e^{-1/t} \}
\]
be the Lebesgue spine, which is thin at the origin
 (see e.g.\ Example~13.4 in \cite{BBbook}).
Let
\[
u(x,t) =\biggl( \frac{e^{2/t}}{t} \biggr)^{1/p} \phi(x,t),
\]
where $\phi\in C^\infty(\R^3 \setm \{0\})$ is such that
\[
    \phi(x,t)=\begin{cases}
      1, & \text{if } |x| <\tfrac{1}{3}e^{-1/t}, \\
      0, & \text{if } |x| >\tfrac{2}{3}e^{-1/t},
      \end{cases}
\]
and we define $u$ arbitrarily at the origin.

Then $u \in \Nploc(\R^3 \setm \{0\})$
and it is easily verified that $u\notin L^p(B)$.
As $\Cp(\{0\})=0$, $u=0$ in $B\setm\itoverline{E}$
and
$(B \setm \itoverline{E}) \cup \{0\}$
is a finely open set containing $0$,
it is still true
that $u \in \Npfloc(B)$, but $u \notin \Nploc(B)=\Npploc(B)$
(by Proposition~\ref{prop-Nploc=>Npploc}\,\ref{j-open}).
\end{example}

\begin{example}
To see that the last inclusion in \eqref{eq-Npfloc} can be strict,
we let $V=B$ be  the unit ball in $\R^n$ with $1 < p \le n$,
and let $u(x)=|x|^\be$ with $\be  \le -n/p$.
Let $V_0$ be any finely open set containing $0$.
Then $0$ is a density point of $V_0$, by Corollary~2.51 in
Mal\'y--Ziemer~\cite{MZ}.
Thus
there are $r_0,\theta>0$ such
that for all $0<r<r_0$,
\begin{equation} \label{eq-density}
    \Lambda_1\bigl( \bigl\{\rho : 0 < \rho < r
      \text{ and } \Lambda_{n-1}(\{x \in V_0 : |x|=\rho\})
       >\theta\rho^{n-1}\bigr\}\bigr)
      > \tfrac{1}{2}r,
\end{equation}
where $\Lambda_d$ denotes $d$-dimensional Lebesgue measure.
Using polar coordinates it is easy to see that $u \notin L^p(B)$.
It then follows from \eqref{eq-density} that $u \notin L^p(V_0)$,
and hence $u \notin \Npfloc(B)$.
On the other hand, as $\Cp(\{0\})=0$, we see that
$u \in \Nploc(B \setm \{0\}) \subset \Npqloc(B \setm \{0\}) = \Npqloc(B)$,
by Lemma~\ref{lem-quasiopen-char}.
\end{example}

\section{Quasicontinuity}
\label{sect-quasicontinuity}

Our aim in this section is to prove Theorem~\ref{thm-Nploc-intro}.
In fact, we prove the following generalization for local Dirichlet spaces.

\begin{thm} \label{thm-Nploc}
Let $U$ be quasiopen and assume that $u \in \Dpqloc(U)$.
Then $u$ is finite q.e.\  and
finely continuous q.e.\ in $U$.
In particular, $u$ is quasicontinuous in $U$.
\end{thm}

For open $U$, quasicontinuity
was deduced in Bj\"orn--Bj\"orn--Shan\-mu\-ga\-lin\-gam~\cite[Theorem~1.1]{BBS5}
for $u \in \Nploc(\Om)$.
To prove Theorem~\ref{thm-Nploc} we will use the following two lemmas.
See Bj\"orn--Bj\"orn--Latvala~\cite{BBL2} for a proof of the first one.

\begin{lem} \label{lem-product}
Let $Y$ be a metric space equipped with a measure which
is bounded on balls.
If $u,v \in \Np(Y)$ are bounded,
then $uv \in \Np(Y)$.
\end{lem}

\begin{lem} \label{lem-finite-qe}
Let $U$ be quasiopen and $u \in \Dpqloc(U)$.
Then $u$ is finite q.e.\ in $U$.
\end{lem}

\begin{proof}
By assumption there is a quasicovering $\B=\{U_j\}_{j=1}^\infty$ of $U$
such that $u \in \Np(U_j)$ for each $j$.
By Lemma~\ref{lem-quasiopen-ppath-meas} and Proposition~\ref{prop-Dp=infty},
 $u$ is finite q.e.\ in $U_j$.
Since $\B$ is a quasicovering, it follows that $u$ is finite q.e.\ in $U$.
\end{proof}

\begin{proof}[Proof of Theorem~\ref{thm-Nploc}]
By assumption there is a quasicovering $\B$ of $U$ such
that $u \in \Dp(\Ut)$ for each $\Ut \in \B$.
In addition we can assume that all $\Ut \in \B$ are bounded.
By Proposition~\ref{prop-strict-quasicover},  for each
$\Ut\in\B$ there exists a \p-strict quasicovering $\B_{\Ut}$ of $\Ut$
consisting of finely open \p-strict subsets $V_{j,\Ut}$
such that $\overline{V}_{j,\Ut} \Subset \Ut$.

First, assume that $u$ is bounded and
let $V=V_{j,\Ut}$ be arbitrary.
Then there is
$v \in \Np_0(\Ut)$ with $v=1$ on $V$
and  $0 \le v \le 1$ everywhere.
Since $u$ and $\Ut$ are bounded and $u\in\Dp(\Ut)$, we get that
$u \in \Np(\Ut)$.
Let $w=v u$, extended by $0$ outside of $\Ut$.
Then $w \in \Np(\Ut)$ by Lemma~\ref{lem-product}.
As $|w|\le Cv\in\Np_0(\Ut)$
(with $C=\sup_U |u|< \infty$), it follows from Lemma~2.37 in \cite{BBbook}
that $w \in \Np_0(\Ut) \subset \Np(X)$,
and in particular $w$ is quasicontinuous in $X$.
By Theorem~4.9\,(b) in Bj\"orn--Bj\"orn--Latvala~\cite{BBL1},
$w$ is finely continuous q.e.\ in $X$.
Since $u = w$ in the finely open set $V$,
it follows that $u$ is finely continuous q.e.\ in $V$,
and thus q.e.\ in $U$, as $V$ was arbitrary.

If $u$ is arbitrary, let
$u_k=\max\{\min\{u,k\},-k\}$, $k=1,2,\ldots$,
be the truncations of $u$ at levels $\pm k$.
By the first part of the proof there is a set $E_k$ such that
$\Cp(E_k)=0$ and $u_k$ is finely continuous at all $x \in U \setm E_k$.

By Lemma~\ref{lem-finite-qe}, there is a set $E_0$ with  $\Cp(E_0)=0$ such
that $u$ is finite  in $U \setm E_0$.
Let $E=\bigcup_{j=0}^\infty E_j$, which is a set with zero capacity.
If $x \in U \setm E$, then $u(x)$ is finite and hence
there is $k$ such that $|u(x)|<k$.
Since $u_k$ is finely continuous at $x$ and $|u(x)|<k$, we
conclude that $u$ is also finely continuous at $x$.
Hence $u$ is finely continuous q.e.\ in $U$.

The quasicontinuity
of $u$ now follows from
Theorem~1.4\,(b) in Bj\"orn--Bj\"orn--Latvala~\cite{BBL2}.
\end{proof}

\begin{proof}[Proof of Theorem~\ref{thm-Nploc-intro}]
This follows directly from 
Proposition~\ref{prop-Nploc=>Npploc}\,\ref{j-quasiopen}
together with Theorem~\ref{thm-Nploc}.
\end{proof}

\section{Sobolev spaces based on fine upper gradients}
\label{sect-Sobolev}

\emph{In this section we assume that $U$ is quasiopen.}

\medskip

The main aim of this section is to prove Theorem~\ref{thm-intro-equiv-Sob}.
To do so it will be convenient to make the following definition,
for reasons that will become clear towards the end of the section.
It has been inspired by the fine gradients in $\R^n$ from
Kilpel\"ainen--Mal\'y~\cite{KiMa92}, see Definition~\ref{def-Wp}
below and the comments after it.
Recall that $g_{u,V}$ is the minimal \p-weak upper gradient of
$u:V\to\eR$ taken with respect to $V$ as the ambient space.

\begin{deff}  \label{def-fine-grad}
A function $\gt_u:U\to[0,\infty]$ is a \emph{\p-fine upper gradient}
of $u \in \Dpqloc(U)$  if there is a quasicovering $\B$ of $U$ such that
$u\in\Dp(V)$ for every $V\in\B$ and $\gt_u=g_{u,V}$ a.e.\ in $V$.
\end{deff}

The following result shows that \p-fine upper gradients
always exist.

\begin{lem} \label{lem-fine-gr-upper-gr}
If $u \in \Dpqloc(U)$,
then it has a unique \p-fine upper gradient $\gt_u$.
\end{lem}

The uniqueness is up to a.e.
Moreover,  by Definition~\ref{def-fine-grad}, if $g:U \to [0,\infty]$ satisfies
$g=\gt_u$ a.e., then $g$ is also a \p-fine upper gradient of $u$.

\begin{proof}
Let $\B$ be a quasicovering of $U$ such that $u \in \Dp(V)$
for every $V \in \B$.
If $V, W \in \B$, then $V \cap W$ is quasiopen and
Lemma~\ref{lem-minimal-restrict} shows that
\[
     g_{u,V} = g_{u,V \cap W} = g_{u,W}
     \quad \text{a.e. in } V \cap W.
\]
We can therefore define $\gt_u:U \to [0,\infty]$ so that
$\gt_u=g_{u,V}$ a.e.\ in $V$ for every $V \in \B$.
By definition, $\gt_u$ is a \p-fine upper gradient of $u$.

To prove the uniqueness, assume that $g$ is any \p-fine upper gradient of $u$,
and let $\B$ and $\B'$ be the quasicoverings given
in Definition~\ref{def-fine-grad} for $\gt_u$ and $g$,
respectively.
Let $V \in \B$ and $W \in \B'$.
Since $V \cap W$ is quasiopen,
Lemma~\ref{lem-minimal-restrict} then yields that
\[
     \gt_u =g_{u,V}=g_{u,V \cap W}=g_{u,W}=g
          \quad \text{a.e.\ in } V \cap W.
\]
As $\B'$ and $\B$ are quasicoverings it follows that
$\gt_u=g$ a.e.\ in $V$, and thus in $U$.
This proves the a.e.\ uniqueness of $\gt_u$.
\end{proof}

Next, we show that \p-fine upper gradients are the same
as minimal \p-weak upper gradients, with minimality in
the appropriate sense.
Note that minimality has
been built into the definition of \p-fine upper gradients.

\begin{thm}\label{thm-minimality}
Let $u \in \Dpqloc(U)$ and let $\gt_u$ be a \p-fine upper gradient of
$u$.
Then $\gt_u \in \Lpqloc(U)$ and it is a \p-weak upper gradient of $u$ in $U$
 which is minimal in the following
two senses\/\textup{:}
\begin{enumerate}
\item \label{min-b}
If $W \subset U$ is quasiopen and $u \in \Dploc(W)$,
then $\gt_u = g_{u,W}$ a.e.\ in $W$.
\item \label{min-a}
If $g \in \Lpqloc(U)$ is a \p-weak upper gradient of $u$, then
$\gt_u \le g$ a.e.
\end{enumerate}
\end{thm}

\begin{proof}
Let $\B$ be the quasicovering of $U$ associated with
$\gt_u$ in Definition~\ref{def-fine-grad}.
It is immediate that $\gt_u \in \Lpqloc(U)$.

Next we show that $\gt_u$ is a \p-weak upper gradient of $u$ in $U$.
Since $\Cp(U\setm \bigcup_{V\in\B} V) =0$ and each $V$ is \p-path open (by
Lemma~\ref{lem-quasiopen-ppath-meas}), we conclude from
Proposition~1.48 in \cite{BBbook}
that \p-almost every curve $\ga$ in $U$ avoids
$U\setm \bigcup_{V\in\B} V$ and is such that  $\ga^{-1}(V)$ is open for each $V\in\B$.
Moreover, by~\cite[Lemma~1.40]{BBbook}, we can assume that
$\gt_u$ is an upper gradient of $u$
(i.e.\ satisfies \eqref{ug-cond})
on
any subcurve $\gat\subset\ga$, whose image is contained in $V$.
Let $\ga:[0,l_\ga]\to \bigcup_{V\in\B} V$ be such a curve.

Since each $\ga^{-1}(V)$ is open, it is a countable union of (relatively) open
subintervals of $[0,l_\ga]$, whose collection for all $V\in\B$
covers $[0,l_\ga]$.
By compactness, $[0,l_\ga]$ can be covered by finitely many
such intervals.
We can then find slightly smaller closed intervals so that
$[0,l_\ga]=\bigcup_{j=1}^N [a_j,b_j]$, where 
$[a_j,b_j]\subset\ga^{-1}(V_j)$ for some $V_j\in\B$.
Since this is a finite union of intervals we can, by decreasing their lengths,
assume that 
$0=a_1< b_1=a_2 < \ldots < b_N =l_\ga $.
We then have
\begin{align*}
|u(\ga(0)) - u(\ga(l_\ga))| &\le \sum_{j=1}^N |u(\ga(a_{j})) - u(\ga(b_{j}))|
\le \sum_{j=1}^N \int_{\ga|_{[a_{j},b_{j}]}} \gt_u\,ds = \int_\ga \gt_u\,ds.
\end{align*}
This shows that $\gt_u$ is a \p-weak upper gradient of $u$ in $U$.

We now turn to the minimality.

\ref{min-b}
If $V \in \B$, then $V \cap W$ is quasiopen. Hence by
Lemma~\ref{lem-minimal-restrict},
\[
     g_{u,W}=g_{u,V \cap W}=g_{u,V} = \gt_u
     \quad \text{a.e. in } V \cap W.
\]
As $\B$ is a quasicovering, we see that $g_{u,W}= \gt_u$ a.e.\ in $W$.

\ref{min-a}
Let $\Bt$ be a quasicovering such that $g \in L^p(\Vt)$ for each $\Vt \in \Bt$.
Then $u \in \Dp(\Vt)$ and by \ref{min-b},
\[
     \gt_u = g_{u,\Vt} \le g
          \quad \text{a.e. in } \Vt.
\]
As $\Bt$ is a quasicovering, it follows that $\gt_u \le g$ a.e.\ in $U$.
\end{proof}

\begin{thm}   \label{thm-Np-iff-gt-Lp}
A function $u$ belongs to $\Np(U)$
if and only if $u\in L^p(U)$
and there exists
a \p-fine upper gradient $\gt_u\in L^p(U)$.
Moreover, in this case, $\gt_u=g_{u,U}$ a.e.\ in $U$.
\end{thm}

Similar statements for $\Nploc$, $\Npploc$, $\Npqloc$, $\Dp$,
$\Dploc$, $\Dpploc$ and $\Dpqloc$ are also valid.

\begin{proof}
If $u \in\Np(U)$, then $u \in L^p(U)$ and the minimal
\p-weak upper gradient $g_{u,U} \in L^p(U)$.
As $u \in \Np(U) \subset \Dpqloc(U)$ it has a \p-fine upper gradient $\gt_u$,
by Lemma~\ref{lem-fine-gr-upper-gr}.
By Theorem~\ref{thm-minimality}, $\gt_u=g_{u,U}$ a.e.,
and thus $\gt_u \in L^p(U)$.

Conversely, if
 $u \in L^p(U)$ and $\gt_u \in L^p(U)$ is a \p-fine upper gradient  $u$,
then $\gt_u$ is a \p-weak upper gradient of $u$ in $U$,
by Theorem~\ref{thm-minimality},
and thus $u \in \Np(U)$.
\end{proof}

The following result shows that in the definition of \p-fine upper gradients,
$u$ can be extended from $V$ to a global Newtonian function.

\begin{prop}  \label{prop-ex-fine-u-g}
If $u \in \Dpqloc(U)$,
then there is a \p-strict quasicovering $\B$
of $U$ such that for every $V\in\B$ there exists
$u_V\in\Np(X)$ with $u=u_V$ in $V$.
\end{prop}

The following result is an immediate corollary.

\begin{cor} \label{cor-Npqloc=Dpqloc}
$\Npqloc(U)=\Dpqloc(U)$.
\end{cor}

\begin{proof}[Proof of Proposition~\ref{prop-ex-fine-u-g}]
Let $\Bt$ be a quasicovering of $U$ such that $u \in \Dp(\Vt)$ for every
$\Vt \in \Bt$.
Theorem~\ref{thm-Nploc} shows that $u$ is quasicontinuous.
Let $U_k=\{x\in U:|u(x)|<k\}$, $k=1,2,\ldots$\,.
By quasicontinuity, each $U_k$ is quasiopen and $\{U_k\}_{k=1}^\infty$ forms
a quasicovering of $U$.
For each $k$, let $\B_k=\{U_k\cap\Vt:\Vt\in\Bt\}$ and for each $W\in\B_k$
use Proposition~\ref{prop-strict-quasicover} to obtain a
\p-strict
quasicovering $\B_W$ of $W$,
together with the associated compactly supported functions in $W$.
Then $\B=\bigcup_{k=1}^\infty \bigcup_{W\in \B_k} \B_W$ is a
\p-strict quasicovering of $U$,
and $u$ is bounded on each $W\in \B$.

Now, let $k$, $W=U_k\cap\Vt\in\B_k$ and $V\in\B_W$ be arbitrary,
and let $v\in\Np_0(W)$ be the associated function
such that $v=1$ in $V$.
Since $u\in \Dp(W)$ is bounded in $W$ and $v$ has compact support in $W$,
we see that $u\in\Np(\spt v)$.
Lemma~\ref{lem-product} implies that $uv \in \Np(\spt v)$, and as
$|uv|\le kv \in\Np_0(\spt v)$, Lemma~2.37
in~\cite{BBbook} then shows that $uv\in\Np_0(\spt v) \subset \Np(X)$.
Let $u_V=uv \in \Np(X)$ for $V \in \B$. Then $u=u_V$ in $V$.
\end{proof}

The following definition is from Kilpel\"ainen--Mal\'y~\cite{KiMa92},
and has been our inspiration for Definition~\ref{def-fine-grad}.
Note however, that the metric space theory allows us to consider $u\in\Dp(V)$
in Definition~\ref{def-fine-grad}, which makes the situation simpler.
In particular, we do not need to go outside of $U$ to define
\p-fine upper gradients.
On the other hand, Proposition~\ref{prop-ex-fine-u-g} shows that one can
equivalently use functions $u_V\in\Np(X)$ such that $u=u_V$ in $V$.
(In \cite{KiMa92} a quasicovering may be uncountable, but it is required
to contain a countable quasicovering in our sense.
For the purpose of the definition below the existence
of a quasicovering in either sense is obviously equivalent.)

\begin{deff}  \label{def-Wp}
Let $U\subset\R^n$ be quasiopen.
A function $u:U\to\R$ belongs to $\Wp(U)$ if

\begin{enumerate}
\renewcommand{\theenumi}{\textup{(\roman{enumi})}}%
\item
$u\in L^p(U)$\textup{;}
\item
there is a
quasicovering $\B$ of $U$ such that for every $V\in\B$
there is an open set $G_V\supset V$ and $u_V\in \Wp(G_V)$
so that $u=u_V$ in $V$\textup{;}
\item
the \emph{fine
gradient} $\grad u$, defined by $\grad u = \grad u_V$ a.e.\ on each $V\in\B$,
also belongs to $L^p(U)$.
\end{enumerate}

\vskip-\bigskipamount 

Moreover, we let
\[
    \|u\|_{\Wp(U)} = \biggl( \int_U (|u|^p + |{\grad u}|^p)\,dx\biggr)^{1/p}.
\]
\end{deff}

As $\B$ is a quasicovering, the gradient $\nabla u$ is well defined a.e.,
and we may pick any representative.
By Proposition~\ref{prop-ex-fine-u-g} and Theorem~\ref{thm-Wp=Np} below,
one can equivalently require that $G_V=\R^n$ in
Definition~\ref{def-Wp}.

Kilpel\"ainen--Mal\'y~\cite{KiMa92} gave this definition
for unweighted $\R^n$, but it makes sense also
in weighted $\R^n$ (with a \p-admissible weight),
provided that we by $\grad u_V$
mean the corresponding weighted Sobolev gradient,
cf.\ the discussion on p.\ 13 in
Heinonen--Kilpel\"ainen--Martio~\cite{HeKiMa}.
Theorem~\ref{thm-Wp=Np} below also holds, with the proof below, for
weighted $\R^n$.

If we change a function in $\Wp(U)$ on a set of measure zero it remains
in $\Wp(U)$, in contrast to the Newtonian case.
To characterize $\Wp(U)$ using
 Newtonian spaces we need to consider the space
\begin{equation} \label{eq-hNp}
        \hNp (U) = \{u : u=v \text{ a.e. for some } v \in \Np(U)\}.
\end{equation}
We also let $\|u\|_{\hNp(U)}=\|v\|_{\Np(U)}$, which is well defined
by Lemma~1.62 in  \cite{BBbook}.

\begin{thm} \label{thm-Wp=Np}
Let $U\subset\R^n$ be a quasiopen set. Then $\hNp(U)=\Wp(U)$
and $\|u\|_{\hNp(U)}= \|u\|_{\Wp(U)}$ for all $u \in \hNp(U)$.

Moreover, $g_{v,U}=|\nabla u|$ a.e.\ in $U$
for all $u \in \hNp(U)$ and $v$
as in \eqref{eq-hNp}.
\end{thm}

By Remark~6.12 in \cite{BBbook} the Sobolev capacity considered
here is the same as in \cite{HeKiMa} and \cite{KiMa92}.

\begin{proof}
Let $u \in \hNp(U)$.
Then there is $v \in \Np(U)$ such that $v=u$ a.e.\ in $U$.
By Proposition~\ref{prop-ex-fine-u-g},
there is a quasicovering $\B$
of $U$ such that for every $V\in\B$ there exists
$v_V\in\Np(\R^n)$ with $v=v_V$ in $V$.
By Proposition~A.12 in \cite{BBbook}, $v_V \in \Wp(\R^n)$
and $g_{v_V}=|\nabla v_V|$ a.e.\ in $\R^n$.
Hence $v \in \Wp(U)$ and $\gt_v=|\nabla v|$ a.e.\ in $U$, where
$\gt_v$ and $\nabla v$ are the \p-fine
and fine gradients of $v$, respectively.
Theorem~\ref{thm-Np-iff-gt-Lp} implies that
$g_{v,U}=\gt_v=|\nabla v|$ a.e.\ in $U$.
As $u=v$ a.e., it follows directly that $u \in \Wp(U)$ and
$\nabla u=\nabla v$ a.e.
Thus,
\[
   \|u\|_{\hNp(U)}=\|v\|_{\Np(U)}=\|v\|_{\Wp(U)}=\|u\|_{\Wp(U)}.
\]

Conversely, let $u \in \Wp(U)$.
Then there is a quasicovering $\B$ with $u_V$ and $G_V$ for $V \in \B$,
as in Definition~\ref{def-Wp}.
By Theorem~4.4 in Heinonen--Kilpel\"ainen--Martio~\cite{HeKiMa},
for each $V \in \B$ there is a quasicontinuous function $\ut_V$ on
the open set $G_V$ such that $\ut_V=u_V$ a.e.\ in $G_V$.
By Proposition~A.13 in \cite{BBbook}, $\ut_V \in \Np(G_V)$.

If $V,W \in \B$, then $\ut_V=\ut_W$ q.e.\ in $G_V \cap G_W$,
by Kilpel\"ainen~\cite{kilp98} (or \cite[Proposition~5.23]{BBbook}).
Hence
we can find functions $\ut: U \to \eR$ and $g:U \to [0,\infty]$ such that
\[
  \ut=\ut_V \text{ q.e.\ in } V
  \quad \text{and} \quad
  g=g_{\ut_V,V} 
  \text{ a.e.\ in } V
\]
for every $V \in \B$.
By definition, $g$ is a \p-fine upper gradient of $\ut$.
Lemma~\ref{lem-minimal-restrict} above and
Proposition~A.12 in \cite{BBbook} yield that
\[
  g=g_{\ut_V,V}=g_{\ut_V,G_V}=|\nabla \ut_V|=|\nabla u_V|=|\nabla u|
  \quad \text{a.e.\ in } V
\]
for every $V \in \B$, and hence
$g=|\grad u|$ a.e.\ in $U$.
Since $|\grad u|\in L^p(U)$, Theorem~\ref{thm-Np-iff-gt-Lp}
implies that $\ut \in \Np(U)$.
As $u=\ut$ a.e.\ in $U$, we get that $u \in \hNp(U)$.
\end{proof}

The following gives a more precise description of which
representatives of $\Wp(U)$ belong to $\Np(U)$ (and not just to
$\hNp(U)$).
As it is valid in metric spaces we state it using $\hNp(U)$, but
in view of Theorem~\ref{thm-Wp=Np}
one can of course
replace $\hNp(U)$ by $\Wp(U)$ if $U \subset \R^n$.

\begin{thm} \label{thm-eqv-Np-qcont-ACCp}
Let $u \in \hNp(U)$.
Then the following are equivalent\/\textup{:}
\begin{enumerate}
\item \label{Wp-item-Np}
    $u \in \Np(U)$\textup{;}
\item \label{Wp-item-qc}
    $u$ is quasicontinuous\/\textup{;}
\item \label{Wp-item-ACCp}
    $u$ is absolutely continuous on \p-almost every curve in $U$.
\end{enumerate}
\end{thm}

\begin{proof}
\ref{Wp-item-Np} \imp \ref{Wp-item-qc}
This follows from Theorem~\ref{thm-Nploc}.

\ref{Wp-item-qc} \imp \ref{Wp-item-Np}
As $u \in \hNp(U)$ there is $v \in \Np(U)$
such that $v=u$ a.e.\ in $U$.
By Theorem~\ref{thm-Nploc}, $v$ is quasicontinuous.
Hence by Kilpel\"ainen~\cite{kilp98}
(or \cite[Proposition~5.23]{BBbook})
$u=v$ q.e., and thus also $u \in \Np(U)$.
Here we also need to appeal to  Proposition~5.22 in \cite{BBbook},
which shows
that our capacity satisfies Kilpel\"ainen's axioms.

\ref{Wp-item-Np} \eqv \ref{Wp-item-ACCp}
This follows from Proposition~1.63 in \cite{BBbook}.
\end{proof}

\begin{proof}[Proof of Theorem~\ref{thm-intro-equiv-Sob}]
This follows directly from Theorems~\ref{thm-Wp=Np} and~\ref{thm-eqv-Np-qcont-ACCp}.
\end{proof}

Let us end by stating the following question.

\begin{openprob} \label{openprob-p-loc}
Let $V \subset \R^n$ be finely open.
Is then
$u \in \Npqloc(V)$ if and only if $u$ is quasicontinuous and
$u \in W^{1,p}_{\textup{\p-loc}}(V)$, as defined by Mal\'y--Ziemer~\cite[p.~149]{MZ}?
\end{openprob}

\end{document}